\documentclass[12pt,fleqn]{article}

\usepackage{amsmath,amsthm,amssymb}
\usepackage{geometry} 
\usepackage[pdfpagemode=UseNone,pdfstartview=FitH]{hyperref}

\newcommand{\Ap}[1][]{A_p\, #1}
\newcommand{\abs}[1]{\left|#1\right|}
\newcommand{\Bp}[1][]{B_p\, #1}
\newcommand{\bdry}[1]{\partial #1}
\newcommand{\closure}[1]{\overline{#1}}

\newcommand{\dint}{\ds{\int}}
\newcommand{\dist}[2]{\text{dist}\, (#1,#2)}
\newcommand{\dnorm}[2][]{\left\|#2\right\|_{#1}^\ast}
\newcommand{\ds}[1]{\displaystyle #1}
\newcommand{\dualp}[3][]{\left(#2,#3\right)_{#1}}
\newcommand{\eps}{\varepsilon}
\newcommand{\F}{{\cal F}}
\newcommand{\half}{\frac{1}{2}}
\newcommand{\M}{{\cal M}}
\newcommand{\minuszero}{\! \setminus\! \set{0}}
\newcommand{\N}{\mathbb N}
\newcommand{\norm}[2][]{\left\|#2\right\|_{#1}}

\renewcommand{\o}{\text{o}}
\newcommand{\PS}[1]{$(\text{PS})_{#1}$}
\newcommand{\pnorm}[2][]{\if #1'' \left|#2\right|_p \else \left|#2\right|_{#1} \fi}

\newcommand{\R}{\mathbb R}

\newcommand{\seq}[1]{\left(#1\right)}
\newcommand{\set}[1]{\left\{#1\right\}}

\newcommand{\Z}{\mathbb Z}

\DeclareMathOperator{\divg}{div}
\DeclareMathOperator{\supp}{supp}

\newtheorem{lemma}{Lemma}[section]

\newtheorem{theorem}[lemma]{Theorem}

\theoremstyle{definition}

\theoremstyle{remark}
\newtheorem{remark}[lemma]{Remark}

\numberwithin{equation}{section}

\title{\bf A multiplicity result for critical elliptic problems involving differences of local and nonlocal operators}
\author{\bf Kanishka Perera\\
Department of Mathematical Sciences\\
Florida Institute of Technology\\
Melbourne, FL 32901, USA\\
\em kperera@fit.edu\\
[\bigskipamount]
\bf Caterina Sportelli\\
Dipartimento di Matematica\\
Università degli Studi di Bari Aldo Moro\\
Via E. Orabona 4, 70125 Bari, Italy\\
\em caterina.sportelli@uniba.it}
\date{}

\begin{document}

\maketitle

\begin{abstract}
We study some critical elliptic problems involving the difference of two nonlocal operators, or the difference of a local operator and a nonlocal operator. The main result is the existence of two nontrivial weak solutions, one with negative energy and the other with positive energy, for all sufficiently small values of a parameter. The proof is based on an abstract result recently obtained in \cite{MR4293883}.
\end{abstract}

\section{Introduction}

Elliptic problems involving sums of nonlocal operators have been recently studied in the literature (see, e.g., \cite{MR3927428, MR3876875, MR3915604, MR3910033, MR4313576, MR4160747, MR4142328} and their references). The purpose of the present paper is to consider critical elliptic problems involving differences of such operators. Here a new phenomenon appears: there exist two nontrivial weak solutions, one with negative energy and the other with positive energy, for all sufficiently small values of a parameter.

We consider the nonlocal problem
\begin{equation} \label{2}
\left\{\begin{aligned}
(- \Delta)_p^s\, u - \mu\, (- \Delta)_q^t\, u & = \lambda\, |u|^{p-2}\, u + |u|^{p_s^\ast - 2}\, u && \text{in } \Omega\\[10pt]
u & = 0 && \text{in } \R^N \setminus \Omega,
\end{aligned}\right.
\end{equation}
where $\Omega$ is a bounded domain in $\R^N$ with Lipschitz boundary, $0 < t < s < 1$, $1 < q < p < N/s$, $(- \Delta)_p^s$ is the fractional $p$-Laplacian operator defined on smooth functions by
\[
(- \Delta)_p^s\, u(x) = 2 \lim_{\eps \searrow 0} \int_{\R^N \setminus B_\eps(x)} \frac{|u(x) - u(y)|^{p-2}\, (u(x) - u(y))}{|x - y|^{N+sp}}\, dy, \quad x \in \R^N,
\]
$p_s^\ast = Np/(N - sp)$ is the fractional critical Sobolev exponent, and $\lambda, \mu > 0$ are parameters. Let $\pnorm{\cdot}$ denote the norm in $L^p(\R^N)$, let
\[
[u]_{s,\,p} = \left(\int_{\R^{2N}} \frac{|u(x) - u(y)|^p}{|x - y|^{N+sp}}\, dx dy\right)^{1/p}
\]
be the Gagliardo seminorm of a measurable function $u : \R^N \to \R$, and let
\[
W^{s,\,p}(\R^N) = \set{u \in L^p(\R^N) : [u]_{s,\,p} < \infty}
\]
be the fractional Sobolev space endowed with the norm
\[
\norm[s,\,p]{u} = \left(\pnorm{u}^p + [u]_{s,\,p}^p\right)^{1/p}.
\]
We work in the closed linear subspace
\[
W^{s,\,p}_0(\Omega) = \set{u \in W^{s,\,p}(\R^N) : u = 0 \text{ a.e.\! in } \R^N \setminus \Omega},
\]
equivalently renormed by setting $\norm{\cdot} = [\cdot]_{s,\,p}$. A weak solution of problem \eqref{2} is a function $u \in W^{s,\,p}_0(\Omega)$ satisfying
\begin{multline*}
\int_{\R^{2N}} \frac{|u(x) - u(y)|^{p-2}\, (u(x) - u(y))\, (v(x) - v(y))}{|x - y|^{N+sp}}\, dx dy\\[7.5pt]
- \mu \int_{\R^{2N}} \frac{|u(x) - u(y)|^{q-2}\, (u(x) - u(y))\, (v(x) - v(y))}{|x - y|^{N+tq}}\, dx dy - \lambda \int_\Omega |u|^{p-2}\, uv\, dx\\[7.5pt]
- \int_\Omega |u|^{p_s^\ast - 2}\, uv\, dx = 0 \quad \forall v \in W^{s,\,p}_0(\Omega).
\end{multline*}
Weak solutions coincide with critical points of the $C^1$-functional
\begin{multline*}
E(u) = \frac{1}{p} \int_{\R^{2N}} \frac{|u(x) - u(y)|^p}{|x - y|^{N+sp}}\, dx dy - \frac{\mu}{q} \int_{\R^{2N}} \frac{|u(x) - u(y)|^q}{|x - y|^{N+tq}}\, dx dy - \frac{\lambda}{p} \int_\Omega |u|^p\, dx\\[7.5pt]
- \frac{1}{p_s^\ast} \int_\Omega |u|^{p_s^\ast}\, dx, \quad u \in W^{s,\,p}_0(\Omega).
\end{multline*}

Let $\sigma((- \Delta)_p^s)$ denote the Dirichlet spectrum of $(- \Delta)_p^s$ on $\Omega$, consisting of those $\lambda \in \R$ for which the eigenvalue problem
\[
\left\{\begin{aligned}
(- \Delta)_p^s\, u & = \lambda\, |u|^{p-2}\, u && \text{in } \Omega\\[10pt]
u & = 0 && \text{in } \R^N \setminus \Omega
\end{aligned}\right.
\]
has a nontrivial weak solution. Let
\[
\dot{W}^{s,\,p}(\R^N) = \set{u \in L^{p_s^\ast}(\R^N) : [u]_{s,\,p} < \infty}
\]
endowed with the norm $\norm{\cdot}$, and let
\begin{equation} \label{1.2}
S = \inf_{u \in \dot{W}^{s,\,p}(\R^N) \setminus \set{0}}\, \frac{\dint_{\R^{2N}} \frac{|u(x) - u(y)|^p}{|x - y|^{N+sp}}\, dx dy}{\left(\dint_{\R^N} |u|^{p_s^\ast}\, dx\right)^{p/p_s^\ast}}
\end{equation}
be the best fractional Sobolev constant. We will prove the following multiplicity result.

\begin{theorem} \label{Theorem 2}
Assume that $0 < t < s < 1$, $1 < q < p < N/s$, $N \ge sp^2$, and $\lambda \in (0,\infty) \setminus \sigma((- \Delta)_p^s)$. Then $\exists\, \mu_0, \kappa > 0$ such that problem \eqref{2} has two nontrivial weak solutions $u_1$ and $u_2$ with
\[
E(u_1) < 0 < E(u_2) < \frac{s}{N}\, S^{N/sp} - \kappa \mu
\]
for all $\mu \in (0,\mu_0)$.
\end{theorem}

When $0 < \lambda < \lambda_1$, where $\lambda_1 > 0$ is the first Dirichlet eigenvalue of $(- \Delta)_p^s$ on $\Omega$, this theorem can be proved using a local minimization argument and the mountain pass theorem. However, when $\lambda > \lambda_1$, the functional $E$ no longer has the mountain pass geometry. In this case, our proof will be based on an abstract multiplicity result recently proved in Perera \cite{MR4293883}, and our solutions are neither local minimizers nor of mountain pass type. They are higher critical points in the sense that they have nontrivial higher critical groups.

\begin{remark}
When $\mu = 0$, one nontrivial weak solution $u$ with
\[
0 < E(u) < \frac{s}{N}\, S^{N/sp}
\]
was obtained in Mosconi et al.\! \cite{MR3530213}.
\end{remark}

\begin{remark}
When $s > 0$ is not an integer and $1 \le q < p \le \infty$, $W^{s,\,p}_0(\Omega) \not\subset W^{s,\,q}_0(\Omega)$ (see Mironescu and Sickel \cite{MR3357858}). So when $t = s$, the functional $E$ is not defined on $W^{s,\,p}_0(\Omega)$ and existence of solutions to problem \eqref{2} is an open problem.
\end{remark}

Theorem \ref{Theorem 2} also holds in the limiting case $s = 1$, where we have the mixed local-nonlocal problem
\begin{equation} \label{1}
\left\{\begin{aligned}
- \Delta_p\, u - \mu\, (- \Delta)_q^t\, u & = \lambda\, |u|^{p-2}\, u + |u|^{p^\ast - 2}\, u && \text{in } \Omega\\[10pt]
u & = 0 && \text{in } \R^N \setminus \Omega,
\end{aligned}\right.
\end{equation}
where $0 < t < 1$, $1 < q < p < N$, $\Delta_p\, u = \divg (|\nabla u|^{p-2}\, \nabla u)$ is the $p$-Laplacian of $u$, $p^\ast = Np/(N - p)$ is the critical Sobolev exponent, and $\lambda, \mu > 0$ are parameters. A weak solution of this problem is a function $u \in W^{1,\,p}_0(\Omega)$ satisfying
\begin{multline*}
\int_\Omega |\nabla u|^{p-2}\, \nabla u \cdot \nabla v\, dx - \mu \int_{\R^{2N}} \frac{|u(x) - u(y)|^{q-2}\, (u(x) - u(y))\, (v(x) - v(y))}{|x - y|^{N+tq}}\, dx dy\\[7.5pt]
- \lambda \int_\Omega |u|^{p-2}\, uv\, dx - \int_\Omega |u|^{p^\ast - 2}\, uv\, dx = 0 \quad \forall v \in W^{1,\,p}_0(\Omega),
\end{multline*}
where functions in $W^{1,\,p}_0(\Omega)$ are extended to be zero outside $\Omega$. Weak solutions coincide with critical points of the $C^1$-functional
\begin{multline*}
E(u) = \frac{1}{p} \int_\Omega |\nabla u|^p\, dx - \frac{\mu}{q} \int_{\R^{2N}} \frac{|u(x) - u(y)|^q}{|x - y|^{N+tq}}\, dx dy - \frac{\lambda}{p} \int_\Omega |u|^p\, dx - \frac{1}{p^\ast} \int_\Omega |u|^{p^\ast}\, dx,\\[7.5pt]
u \in W^{1,\,p}_0(\Omega).
\end{multline*}
Let $\sigma(- \Delta_p)$ denote the Dirichlet spectrum of $- \Delta_p$ on $\Omega$, consisting of those $\lambda \in \R$ for which the eigenvalue problem
\[
\left\{\begin{aligned}
- \Delta_p\, u & = \lambda\, |u|^{p-2}\, u && \text{in } \Omega\\[10pt]
u & = 0 && \text{on } \bdry{\Omega}
\end{aligned}\right.
\]
has a nontrivial weak solution, and let
\[
S = \inf_{u \in W^{1,\,p}_0(\Omega) \setminus \set{0}}\, \frac{\dint_\Omega |\nabla u|^p\, dx}{\left(\dint_\Omega |u|^{p^\ast} dx\right)^{p/p^\ast}}
\]
be the best Sobolev constant. We have the following multiplicity result.

\begin{theorem} \label{Theorem 1}
Assume that $0 < t < 1$, $1 < q < p < N$, $N \ge p^2$, and $\lambda \in (0,\infty) \setminus \sigma(- \Delta_p)$. Then $\exists\, \mu_0, \kappa > 0$ such that problem \eqref{1} has two nontrivial weak solutions $u_1$ and $u_2$ with
\[
E(u_1) < 0 < E(u_2) < \frac{1}{N}\, S^{N/p} - \kappa \mu
\]
for all $\mu \in (0,\mu_0)$.
\end{theorem}

\begin{remark}
When $\mu = 0$, one nontrivial weak solution $u$ with
\[
0 < E(u) < \frac{1}{N}\, S^{N/p}
\]
was obtained in Garc{\'{\i}}a Azorero and Peral Alonso \cite{MR912211}, Egnell \cite{MR956567}, Guedda and V{\'e}ron \cite{MR1009077}, Arioli and Gazzola \cite{MR1741848}, and Degiovanni and Lancelotti \cite{MR2514055}.
\end{remark}

\begin{remark}
When $t = 1$, the noncompactness of the embedding $W^{1,\,p}_0(\Omega) \subset W^{1,\,q}_0(\Omega)$ makes it difficult to verify that the functional $E$ satisfies the \PS{} condition and existence of solutions to problem \eqref{1} is an open problem.
\end{remark}

\begin{remark}
A mixed problem arising in the description of an ecological niche subject to local and nonlocal dispersals was recently considered in Dipierro and Valdinoci \cite{MR4249816}.
\end{remark}

In the next section we will recall the abstract result from Perera \cite{MR4293883} that we will need to prove Theorem \ref{Theorem 2}. In the last section we will give the proof of Theorem \ref{Theorem 2}. Proof of Theorem \ref{Theorem 1} is similar and therefore omitted.

\section{Abstract multiplicity result}

Let $(W,\norm{\, \cdot\,})$ be a uniformly convex Banach space with dual $(W^\ast,\dnorm{\, \cdot\,})$ and duality pairing $\dualp{\cdot}{\cdot}$. Recall that $h \in C(W,W^\ast)$ is a potential operator if there is a functional $H \in C^1(W,\R)$, called a potential for $h$, such that $H' = h$. We consider the nonlinear operator equation
\begin{equation} \label{3}
\Ap[u] - \mu f(u) = \lambda \Bp[u] + g(u)
\end{equation}
in $W^\ast$, where $\Ap, \Bp, f, g \in C(W,W^\ast)$ are potential operators satisfying the following assumptions and $\lambda, \mu > 0$ are parameters:
\begin{enumerate}
\item[$(A_1)$] $\Ap$ is $(p - 1)$-homogeneous and odd for some $p \in (1,\infty)$, i.e., $\Ap[(tu)] = |t|^{p-2}\, t\, \Ap[u]$ for all $u \in W$ and $t \in \R$,
\item[$(A_2)$] $\dualp{\Ap[u]}{v} \le \norm{u}^{p-1} \norm{v}$ for all $u, v \in W$, and equality holds if and only if $\alpha u = \beta v$ for some $\alpha, \beta \ge 0$, not both zero (in particular, $\dualp{\Ap[u]}{u} = \norm{u}^p$ for all $u \in W$),
\item[$(B_1)$] $\Bp$ is $(p - 1)$-homogeneous and odd, i.e., $\Bp[(tu)] = |t|^{p-2}\, t\, \Bp[u]$ for all $u \in W$ and $t \in \R$,
\item[$(B_2)$] $\dualp{\Bp[u]}{u} > 0$ for all $u \in W \setminus \set{0}$, and $\dualp{\Bp[u]}{v} \le \dualp{\Bp[u]}{u}^{(p-1)/p} \dualp{\Bp[v]}{v}^{1/p}$ for all $u, v \in W$,
\item[$(B_3)$] $\Bp$ is a compact operator,
\item[$(F_1)$] the potential $F$ of $f$ with $F(0) = 0$ satisfies $\ds{\lim_{t \to 0}}\, \dfrac{F(tu)}{|t|^p} = + \infty$ uniformly on compact subsets of $W \setminus \set{0}$,
\item[$(F_2)$] $F(u) > 0$ for all $u \in W \setminus \set{0}$,
\item[$(F_3)$] $F$ is bounded on bounded subsets of $W$,
\item[$(G_1)$] the potential $G$ of $g$ with $G(0) = 0$ satisfies $G(u) = \o(\norm{u}^p)$ as $u \to 0$,
\item[$(G_2)$] $G(u) > 0$ for all $u \in W \setminus \set{0}$,
\item[$(G_3)$] $G$ is bounded on bounded subsets of $W$,
\item[$(G_4)$] $\ds{\lim_{t \to + \infty}} \dfrac{G(tu)}{t^p} = + \infty$ uniformly on compact subsets of $W \setminus \set{0}$.
\end{enumerate}
Solutions of equation \eqref{3} coincide with critical points of the $C^1$-functional
\[
E(u) = I_p(u) - \lambda J_p(u) - \mu F(u) - G(u), \quad u \in W,
\]
where
\[
I_p(u) = \frac{1}{p} \dualp{\Ap[u]}{u} = \frac{1}{p} \norm{u}^p, \qquad J_p(u) = \frac{1}{p} \dualp{\Bp[u]}{u}
\]
are the potentials of $\Ap$ and $\Bp$ satisfying $I_p(0) = J_p(0) = 0$, respectively (see \cite[Proposition 3.1]{MR4293883}).

Let $\M = \set{u \in W : I_p(u) = 1}$. Then $\M \subset W \setminus \set{0}$ is a bounded complete symmetric $C^1$-Finsler manifold radially homeomorphic to the unit sphere in $W$, and eigenvalues of the nonlinear eigenvalue problem
\begin{equation} \label{4}
\Ap[u] = \lambda \Bp[u]
\end{equation}
coincide with critical values of the $C^1$-functional
\[
\Psi(u) = \frac{1}{J_p(u)}, \quad u \in \M.
\]
Denote by $\F$ the class of symmetric subsets of $\M$ and by $i(M)$ the $\Z_2$-cohomological index of $M \in \F$ (see Fadell and Rabinowitz \cite{MR0478189}). For $k \in \N$, let $\F_k = \set{M \in \F : i(M) \ge k}$ and set
\[
\lambda_k := \inf_{M \in \F_k}\, \sup_{u \in M}\, \Psi(u).
\]
Then $\lambda_1 = \inf \Psi(\M) > 0$ is the first eigenvalue of problem \eqref{4} and $\lambda_1 \le \lambda_2 \le \cdots$ is an unbounded sequence of eigenvalues (see Perera \cite{MR1998432}). Moreover, denoting by $\Psi^a = \set{u \in \M : \Psi(u) \le a}$ (resp. $\Psi_a = \set{u \in \M : \Psi(u) \ge a}$) the sublevel (resp. superlevel) sets of $\Psi$, we have
\[
\lambda_k < \lambda_{k+1} \implies i(\Psi^{\lambda_k}) = i(\M \setminus \Psi_{\lambda_{k+1}}) = k
\]
(see Perera et al.\! \cite[Theorem 4.6]{MR2640827}). The following multiplicity result was proved in Perera \cite{MR4293883}.

\begin{theorem}[{\cite[Theorem 1.7]{MR4293883}}] \label{Theorem 3}
Assume that $(A_1)$--$(G_4)$ hold, $\exists\, c_\mu > 0$ such that $E$ satisfies the {\em \PS{c}} condition for all $c < c_\mu$, $\lambda_k \le \lambda < \lambda_{k+1}$, and there exist a compact symmetric subset $C$ of $\Psi^\lambda$ with $i(C) = k$ and $w_0 \in \M \setminus C$ such that
\begin{equation} \label{13}
\sup_{v \in C,\, \sigma, \tau \ge 0}\, E(\sigma v + \tau w_0) < c_\mu
\end{equation}
for all sufficiently small $\mu > 0$. Then $\exists\, \mu_0 > 0$ such that equation \eqref{3} has two nontrivial solutions $u_1$ and $u_2$ with
\[
E(u_1) < 0 < E(u_2) < c_\mu
\]
for all $\mu \in (0,\mu_0)$.
\end{theorem}

In the next section we will apply Theorem \ref{Theorem 3} to prove Theorem \ref{Theorem 2}.

\section{Proof of Theorem \ref{Theorem 2}}

We apply Theorem \ref{Theorem 3} with $W = W^{s,\,p}_0(\Omega)$ and the operators $\Ap, \Bp, f, g \in C(W^{s,\,p}_0(\Omega),\linebreak W^{-s,\,p'}(\Omega))$ given by
\begin{multline*}
\dualp{\Ap[u]}{v} = \int_{\R^{2N}} \frac{|u(x) - u(y)|^{p-2}\, (u(x) - u(y))\, (v(x) - v(y))}{|x - y|^{N+sp}}\, dx dy,\\[5pt]
\dualp{\Bp[u]}{v} = \int_\Omega |u|^{p-2}\, uv\, dx,\\[5pt]
\dualp{f(u)}{v} = \int_{\R^{2N}} \frac{|u(x) - u(y)|^{q-2}\, (u(x) - u(y))\, (v(x) - v(y))}{|x - y|^{N+tq}}\, dx dy,\\[5pt]
\dualp{g(u)}{v} = \int_\Omega |u|^{p_s^\ast - 2}\, uv\, dx, \quad u, v \in W^{s,\,p}_0(\Omega).
\end{multline*}
It is easily seen that $(A_1)$--$(G_4)$ hold.

First we determine a threshold level below which the functional $E$ satisfies the \PS{} condition.

\begin{lemma} \label{Lemma 1}
Let $\mu \in (0,1)$. Then $\exists\, \kappa > 0$ such that $E$ satisfies the {\em \PS{c}} condition for all
\begin{equation} \label{5}
c < \frac{s}{N}\, S^{N/sp} - \kappa \mu.
\end{equation}
\end{lemma}

\begin{proof}
Let $c \in \R$ and let $\seq{u_j}$ be a sequence in $W^{s,\,p}_0(\Omega)$ such that
\begin{multline} \label{6}
E(u_j) = \frac{1}{p} \int_{\R^{2N}} \frac{|u_j(x) - u_j(y)|^p}{|x - y|^{N+sp}}\, dx dy - \frac{\mu}{q} \int_{\R^{2N}} \frac{|u_j(x) - u_j(y)|^q}{|x - y|^{N+tq}}\, dx dy - \frac{\lambda}{p} \int_\Omega |u_j|^p\, dx\\[7.5pt]
- \frac{1}{p_s^\ast} \int_\Omega |u_j|^{p_s^\ast}\, dx = c + \o(1)
\end{multline}
and
\begin{multline} \label{7}
\dualp{E'(u_j)}{v} = \int_{\R^{2N}} \frac{|u_j(x) - u_j(y)|^{p-2}\, (u_j(x) - u_j(y))\, (v(x) - v(y))}{|x - y|^{N+sp}}\, dx dy\\[7.5pt]
- \mu \int_{\R^{2N}} \frac{|u_j(x) - u_j(y)|^{q-2}\, (u_j(x) - u_j(y))\, (v(x) - v(y))}{|x - y|^{N+tq}}\, dx dy - \lambda \int_\Omega |u_j|^{p-2}\, uv\, dx\\[7.5pt]
- \int_\Omega |u_j|^{p_s^\ast - 2}\, uv\, dx = \o(\norm{v}) \quad \forall v \in W^{s,\,p}_0(\Omega).
\end{multline}
Taking $v = u_j$ in \eqref{7} gives
\begin{multline} \label{8}
\int_{\R^{2N}} \frac{|u_j(x) - u_j(y)|^p}{|x - y|^{N+sp}}\, dx dy - \mu \int_{\R^{2N}} \frac{|u_j(x) - u_j(y)|^q}{|x - y|^{N+tq}}\, dx dy - \lambda \int_\Omega |u_j|^p\, dx\\[7.5pt]
- \int_\Omega |u_j|^{p_s^\ast}\, dx = \o(\norm{u_j}).
\end{multline}
Fix $r \in (p,p_s^\ast)$. Dividing \eqref{8} by $r$ and subtracting from \eqref{6} gives
\[
\left(\frac{1}{p} - \frac{1}{r}\right) \norm{u_j}^p - \mu \left(\frac{1}{q} - \frac{1}{r}\right) [u_j]_{t,\,q}^q - \lambda \left(\frac{1}{p} - \frac{1}{r}\right) \pnorm{u_j}^p + \left(\frac{1}{r} - \frac{1}{p_s^\ast}\right) \pnorm[p_s^\ast]{u_j}^{p_s^\ast} = c + \o(1) + \o(\norm{u_j}).
\]
Since $t < s$ and $q < p < r < p_s^\ast$, this implies that $\exists\, a_0 > 0$, independent of $\mu \in (0,1)$, such that $\norm{u_j} \le a_0$. So a renamed subsequence of $\seq{u_j}$ converges to some $u$ weakly in $W^{s,\,p}_0(\Omega)$, strongly in $W^{t,\,q}(\Omega)$ and in $L^p(\Omega)$, and a.e.\! in $\Omega$. Moreover,
\[
\norm{u} \le \liminf_{j \to \infty}\, \norm{u_j} \le a_0.
\]
Setting $\widetilde{u}_j = u_j - u$, we will show that $\widetilde{u}_j \to 0$ in $W^{s,\,p}_0(\Omega)$.

Equation \eqref{8} implies
\begin{equation} \label{9}
\norm{u_j}^p = \pnorm[p_s^\ast]{u_j}^{p_s^\ast} + \mu\, [u]_{t,\,q}^q + \lambda \pnorm{u}^p + \o(1),
\end{equation}
where $\pnorm[p_s^\ast]{\cdot}$ denotes the $L^{p_s^\ast}\!(\Omega)$-norm. Taking $v = u$ in \eqref{7} and passing to the limit gives
\begin{equation} \label{10}
\norm{u}^p = \pnorm[p_s^\ast]{u}^{p_s^\ast} + \mu\, [u]_{t,\,q}^q + \lambda \pnorm{u}^p.
\end{equation}
Since
\begin{equation} \label{11}
\norm{\widetilde{u}_j}^p = \norm{u_j}^p - \norm{u}^p + \o(1)
\end{equation}
and
\[
\pnorm[p_s^\ast]{\widetilde{u}_j}^{p_s^\ast} = \pnorm[p_s^\ast]{u_j}^{p_s^\ast} - \pnorm[p_s^\ast]{u}^{p_s^\ast} + \o(1)
\]
by the Br{\'e}zis-Lieb lemma \cite[Theorem 1]{MR699419}, \eqref{9} and \eqref{10} imply
\[
\norm{\widetilde{u}_j}^p = \pnorm[p_s^\ast]{\widetilde{u}_j}^{p_s^\ast} + \o(1) \le \frac{\norm{\widetilde{u}_j}^{p_s^\ast}}{S^{p_s^\ast/p}} + \o(1),
\]
so
\begin{equation} \label{12}
\norm{\widetilde{u}_j}^p \left(S^{N/(N-sp)} - \norm{\widetilde{u}_j}^{sp^2/(N-sp)}\right) \le \o(1).
\end{equation}
On the other hand, \eqref{6} implies
\[
c = \frac{1}{p} \norm{u_j}^p - \frac{1}{p_s^\ast} \pnorm[p_s^\ast]{u_j}^{p_s^\ast} - \frac{\mu}{q}\, [u]_{t,\,q}^q - \frac{\lambda}{p} \pnorm{u}^p + \o(1),
\]
and a straightforward calculation combining this with \eqref{9}--\eqref{11} gives
\[
c = \frac{s}{N} \norm{\widetilde{u}_j}^p + \frac{s}{N} \pnorm[p_s^\ast]{u}^{p_s^\ast} - \mu \left(\frac{1}{q} - \frac{1}{p}\right) [u]_{t,\,q}^q + \o(1).
\]
Since $\norm{u} \le a_0$, this gives
\[
\norm{\widetilde{u}_j}^p \le \frac{N}{s}\, (c + \kappa \mu) + \o(1)
\]
for some constant $\kappa > 0$. Combining this with \eqref{12} shows that $\widetilde{u}_j \to 0$ when \eqref{5} holds.
\end{proof}

As we have noted in the introduction, the result follows from a local minimization argument and the mountain pass theorem when $0 < \lambda < \lambda_1$, so suppose $\lambda > \lambda_1$. Since $\lambda \notin \sigma((- \Delta)_p^s)$, then $\lambda_k < \lambda < \lambda_{k+1}$ for some $k \ge 1$. We apply Theorem \ref{Theorem 3} with
\[
c_\mu = \frac{s}{N}\, S^{N/sp} - \kappa \mu,
\]
where $\kappa > 0$ is as in Lemma \ref{Lemma 1}. We have $\M = \set{u \in W^{s,\,p}_0(\Omega) : \norm{u}^p = p}$ and $\Psi(u) = p/\int_\Omega |u|^p\, dx$ for $u \in \M$. Let
\[
E_0(u) = \frac{1}{p} \int_{\R^{2N}} \frac{|u(x) - u(y)|^p}{|x - y|^{N+sp}}\, dx dy - \frac{\lambda}{p} \int_\Omega |u|^p\, dx - \frac{1}{p_s^\ast} \int_\Omega |u|^{p_s^\ast}\, dx, \quad u \in W^{s,\,p}_0(\Omega).
\]
We will show that there exist a compact symmetric subset $C$ of $\Psi^\lambda$ with $i(C) = k$ and $w_0 \in \M \setminus C$ such that
\begin{equation} \label{2.58}
\sup_{v \in C,\, \sigma, \tau \ge 0}\, E_0(\sigma v + \tau w_0) < \frac{s}{N}\, S^{N/sp}.
\end{equation}
Since $E(u) \le E_0(u)$ for all $u \in W^{s,\,p}_0(\Omega)$, then \eqref{13} will hold if $\mu > 0$ is sufficiently small, so the desired conclusion will follow.

We may assume without loss of generality that $0 \in \Omega$. Let $\delta_0 = \dist{0}{\bdry{\Omega}}$. Since $\lambda_k < \lambda_{k+1}$, $\Psi^{\lambda_k}$ has a compact symmetric subset $C_0$ of index $k$ that is bounded in $L^\infty(\Omega)$ (see Mosconi et al.\! \cite[Proposition 3.1]{MR3530213}). Let $\eta : [0,\infty) \to [0,1]$ be a smooth function such that $\eta(t) = 0$ for $t \le 3/4$ and $\eta(t) = 1$ for $t \ge 1$, let
\[
u_\delta(x) = \eta\bigg(\frac{|x|}{\delta}\bigg)\, u(x), \quad u \in C_0,\, 0 < \delta \le \delta_0/2,
\]
and set
\[
C = \set{\pi_\M(u_\delta) : u \in C_0},
\]
where $\pi_\M : W^{s,\,p}_0(\Omega) \minuszero \to \M,\, u \mapsto p^{1/p}\, u/\norm{u}$ is the radial projection onto $\M$. The following lemma was proved in Mosconi et al.\! \cite{MR3530213}.

\begin{lemma}[Mosconi et al.\! {\cite[Proposition 3.2]{MR3530213}}] \label{Lemma 2.10}
The set $C$ is a compact symmetric subset of $\Psi^{\lambda_k + c_1\, \delta^{N-sp}}$ for some constant $c_1 > 0$ and is bounded in $L^\infty(\Omega)$. If $\lambda_k + c_1\, \delta^{N-sp} < \lambda_{k+1}$, then $i(C) = k$.
\end{lemma}

Let $\delta \in (0,\delta_0/2]$ be so small that $\lambda_k + c_1\, \delta^{N-sp} < \lambda$. Then $C$ is a compact symmetric subset of $\Psi^\lambda$ with $i(C) = k$ that is bounded in $L^\infty(\Omega)$ by Lemma \ref{Lemma 2.10}. We will show that $\exists\, w_0 \in \M \setminus C$ with support in $\closure{B_{\delta/2}(0)}$ such that \eqref{2.58} holds.

Recall that there exists a nonnegative, radially symmetric, and decreasing minimizer $U(x) = U(r),\, r = |x|$ for $S$ satisfying
\[
\int_{\R^{2N}} \frac{|U(x) - U(y)|^p}{|x - y|^{N+sp}}\, dx dy = \int_{\R^N} U(x)^{p_s^\ast}\, dx = S^{N/sp}
\]
and
\begin{equation} \label{2.60}
c_2\, r^{-(N-sp)/(p-1)} \le U(r) \le c_3\, r^{-(N-sp)/(p-1)} \quad \forall r \ge 1
\end{equation}
for some constants $c_2, c_3 > 0$ (see Brasco et al.\! \cite{MR3461371}). Then the functions
\[
u_\eps(x) = \eps^{-(N-sp)/p}\, U\bigg(\frac{|x|}{\eps}\bigg), \quad \eps > 0
\]
are also minimizers for $S$ satisfying
\[
\int_{\R^{2N}} \frac{|u_\eps(x) - u_\eps(y)|^p}{|x - y|^{N+sp}}\, dx dy = \int_{\R^N} u_\eps(x)^{p_s^\ast}\, dx = S^{N/sp}
\]
and
\[
\frac{U(\theta r)}{U(r)} \le \frac{c_3}{c_2}\, \theta^{-(N-sp)/(p-1)} \le \half \quad \forall r \ge 1
\]
if $\theta > 1$ is a sufficiently large constant. Let
\[
u_{\eps,\delta}(x) = \begin{cases}
u_\eps(x) & \text{if } |x| \le \delta\\[10pt]
\dfrac{u_\eps(\delta)\, (u_\eps(x) - u_\eps(\theta \delta))}{u_\eps(\delta) - u_\eps(\theta \delta)} & \text{if } \delta < |x| < \theta \delta\\[10pt]
0 & \text{if } |x| \ge \theta \delta
\end{cases}
\]
and
\[
w_{\eps,\delta}(x) = \frac{u_{\eps,\delta}(x)}{\left(\dint_{\R^N} u_{\eps,\delta}(x)^{p_s^\ast}\, dx\right)^{1/p_s^\ast}}
\]
for $0 < \delta \le \delta_0/2$. Then
\begin{equation} \label{2.61}
\int_{\R^N} w_{\eps,\delta}(x)^{p_s^\ast}\, dx = 1
\end{equation}
and for $\eps \le \delta/2$ we have the estimates
\begin{gather}
\label{2.62} \int_{\R^{2N}} \frac{|w_{\eps,\delta}(x) - w_{\eps,\delta}(y)|^p}{|x - y|^{N+sp}}\, dx dy \le S + c_4\, (\eps/\delta)^{(N-sp)/(p-1)},\\[10pt]
\label{2.63} \int_{\R^N} w_{\eps,\delta}^p(x)\, dx \ge \begin{cases}
c_5\, \eps^{sp} & \text{if } N > sp^2\\[10pt]
c_5\, \eps^{sp} \abs{\log\, (\eps/\delta)} & \text{if } N = sp^2
\end{cases}
\end{gather}
for some constants $c_4, c_5 > 0$ (see Mosconi et al.\! \cite[Lemma 2.7]{MR3530213}). Let
\[
w_0 = \pi_\M(w_{\eps,\delta/2 \theta}).
\]
Since functions in $C$ have their supports in $\Omega \setminus B_{3 \delta/4}(0)$, while the support of $w_0$ is in $\closure{B_{\delta/2}(0)}$, $w_0 \in \M \setminus C$. We will show that \eqref{2.58} holds if $\eps, \delta > 0$ are sufficiently small.

Note that \eqref{2.58} is equivalent to
\begin{equation} \label{2.64}
\sup_{v \in C,\, \sigma, \tau \ge 0}\, E_0(\sigma v + \tau w_{\eps,\delta/2 \theta}) < \frac{s}{N}\, S^{N/sp}.
\end{equation}
Since each $v \in C$ and $w_{\eps,\delta/2 \theta}$ have disjoint supports,
\begin{multline} \label{2.65}
E_0(\sigma v + \tau w_{\eps,\delta/2 \theta}) = \frac{1}{p} \int_{\R^{2N}} \frac{|(\sigma v(x) + \tau w_{\eps,\delta/2 \theta}(x)) - (\sigma v(y) + \tau w_{\eps,\delta/2 \theta}(y))|^p}{|x - y|^{N+sp}}\, dx dy\\[10pt]
- \int_\Omega \left(\frac{\lambda \sigma^p}{p}\, |v|^p + \frac{\sigma^{p_s^\ast}}{p_s^\ast}\, |v|^{p_s^\ast}\right) dx - \int_\Omega \left(\frac{\lambda \tau^p}{p}\, w_{\eps,\delta/2 \theta}^p + \frac{\tau^{p_s^\ast}}{p_s^\ast}\, w_{\eps,\delta/2 \theta}^{p_s^\ast}\right) dx.
\end{multline}
Denote by $I$ the first integral on the right-hand side and set $\alpha = N - sp > 0$.

\begin{lemma} \label{Lemma 2.12}
For $v \in C$ and $\sigma, \tau \ge 0$,
\begin{multline*}
I \le \sigma^p \left(\int_{\R^{2N}} \frac{|v(x) - v(y)|^p}{|x - y|^{N+sp}}\, dx dy + p\, c_6\, \delta^\alpha\right) + \tau^p\, \bigg(\int_{\R^{2N}} \frac{|w_{\eps,\delta/2 \theta}(x) - w_{\eps,\delta/2 \theta}(y)|^p}{|x - y|^{N+sp}}\, dx dy\\[10pt]
+ c_7\, (\eps/\delta)^{\alpha/(p-1)}\bigg)
\end{multline*}
for some constants $c_6, c_7 > 0$.
\end{lemma}

\begin{proof}
Since $\supp v \subset B_{3 \delta/4}^c$ and $\supp w_{\eps,\delta/2 \theta} \subset \closure{B_{\delta/2}}$,
\begin{multline} \label{2.66}
I \le \sigma^p \int_{B_{\delta/2}^c \times B_{\delta/2}^c} \frac{|v(x) - v(y)|^p}{|x - y|^{N+sp}}\, dx dy + \tau^p \int_{B_{3 \delta/4} \times B_{3 \delta/4}} \frac{|w_{\eps,\delta/2 \theta}(x) - w_{\eps,\delta/2 \theta}(y)|^p}{|x - y|^{N+sp}}\, dx dy\\[10pt]
+ 2 \int_{B_{3 \delta/4}^c \times B_{\delta/2}} \frac{|\sigma v(x) - \tau w_{\eps,\delta/2 \theta}(y)|^p}{|x - y|^{N+sp}}\, dx dy =: \sigma^p I_1 + \tau^p I_2 + 2 I_3.
\end{multline}

First suppose $p \ge 2$. To estimate $I_3$, we use the elementary inequality
\[
|a + b|^p \le |a|^p + |b|^p + C_p\, (|a|^{p-1} |b| + |a||b|^{p-1}) \quad \forall a, b \in \R
\]
for some constant $C_p > 0$. Since $v(y) = 0$ for $y \in B_{\delta/2}$ and $w_{\eps,\delta/2 \theta}(x) = 0$ for $x \in B_{3 \delta/4}^c$, we get
\begin{multline} \label{2.67}
I_3 \le \sigma^p \int_{B_{3 \delta/4}^c \times B_{\delta/2}} \frac{|v(x) - v(y)|^p}{|x - y|^{N+sp}}\, dx dy + \tau^p \int_{B_{3 \delta/4}^c \times B_{\delta/2}} \frac{|w_{\eps,\delta/2 \theta}(x) - w_{\eps,\delta/2 \theta}(y)|^p}{|x - y|^{N+sp}}\, dx dy\\[10pt]
\hspace{-22.61pt} + C_p\, \Bigg(\sigma^{p-1} \tau \int_{B_{3 \delta/4}^c \times B_{\delta/2}} \frac{|v(x)|^{p-1}\, w_{\eps,\delta/2 \theta}(y)}{|x - y|^{N+sp}}\, dx dy + \sigma \tau^{p-1} \int_{B_{3 \delta/4}^c \times B_{\delta/2}} \frac{|v(x)|\, w_{\eps,\delta/2 \theta}(y)^{p-1}}{|x - y|^{N+sp}}\, dx dy\Bigg)\\[10pt]
=: \sigma^p I_4 + \tau^p I_5 + C_p \left(\sigma^{p-1} \tau J_1 + \sigma \tau^{p-1} J_{p-1}\right),
\end{multline}
where
\[
J_q = \int_{B_{3 \delta/4}^c \times B_{\delta/2}} \frac{|v(x)|^{p-q}\, w_{\eps,\delta/2 \theta}(y)^q}{|x - y|^{N+sp}}\, dx dy, \quad q = 1, p - 1.
\]
Since $C$ is bounded in $L^\infty(\Omega)$ and
\[
|x - y| \ge |x| - |y| > |x| - \frac{\delta}{2} \ge |x| - \frac{2}{3}\, |x| = \frac{|x|}{3} \quad \forall (x,y) \in B_{3 \delta/4}^c \times B_{\delta/2},
\]
we have
\begin{equation} \label{2.68}
J_q \le c_8 \int_{B_{3 \delta/4}^c \times B_{\delta/2}} \frac{w_{\eps,\delta/2 \theta}(y)^q}{|x|^{N+sp}}\, dx dy = c_9\, \delta^{-sp} \int_{B_{\delta/2}} w_{\eps,\delta/2 \theta}(y)^q\, dy
\end{equation}
for some constants $c_8, c_9 > 0$. By Mosconi et al.\! \cite[Lemma 2.7]{MR3530213}, $|u_{\eps,\delta/2 \theta}|_{p_s^\ast}$ is bounded away from zero and hence
\begin{equation} \label{2.69}
\int_{B_{\delta/2}} w_{\eps,\delta/2 \theta}(y)^q\, dy \le c_{10} \int_{B_{\delta/2}} u_{\eps,\delta/2 \theta}(y)^q\, dy
\end{equation}
for some constant $c_{10} > 0$. Noting that $u_{\eps,\delta/2 \theta} \le u_\eps$, we have
\begin{multline} \label{2.70}
\int_{B_{\delta/2}} u_{\eps,\delta/2 \theta}(y)^q\, dy \le \int_{B_{\delta/2}} u_\eps(y)^q\, dy = \eps^{- \alpha q/p} \int_{B_{\delta/2}} U\bigg(\frac{|y|}{\eps}\bigg)^q\, dy\\[10pt]
= \eps^{N - \alpha q/p} \int_{B_{\delta/2 \eps}} U(|y|)^q\, dy.
\end{multline}
When $q < N(p - 1)/\alpha$, \eqref{2.60} gives
\begin{equation} \label{2.71}
\int_{B_{\delta/2 \eps}} U(|y|)^q\, dy \le c_{11}\, (\delta/\eps)^{N - \alpha q/(p-1)}
\end{equation}
for some constant $c_{11} > 0$, in particular, \eqref{2.71} holds for $q = 1$ when $p > 2N/(N + s)$ and for $q = p - 1$. Combining \eqref{2.68}--\eqref{2.71} gives
\[
J_q \le c_{12}\, \delta^{\alpha\, (p-q-1)/(p-1)}\, \eps^{\alpha q/p\,(p-1)}
\]
for some constant $c_{12} > 0$. So
\[
\sigma^{p-q}\, \tau^q J_q \le c_{12} \left(\delta^\alpha \sigma^p\right)^{1-q/p} \left((\eps/\delta)^{\alpha/(p-1)}\, \tau^p\right)^{q/p} \le c_{13}\, \delta^\alpha \sigma^p + c_{14}\, (\eps/\delta)^{\alpha/(p-1)}\, \tau^p
\]
for some constants $c_{13}, c_{14} > 0$ by Young's inequality. Combining this with \eqref{2.66} and \eqref{2.67}, and noting that
\[
I_1 + 2 I_4 = \int_{\R^{2N}} \frac{|v(x) - v(y)|^p}{|x - y|^{N+sp}}\, dx dy, \qquad I_2 + 2 I_5 = \int_{\R^{2N}} \frac{|w_{\eps,\delta/2 \theta}(x) - w_{\eps,\delta/2 \theta}(y)|^p}{|x - y|^{N+sp}}\, dx dy,
\]
we get the desired conclusion in this case.

If $1 < p < 2$, we use the elementary inequality
\[
|a + b|^p \le |a|^p + |b|^p + p\, |a||b|^{p-1} \quad \forall a, b \in \R
\]
to get
\[
I_3 \le \sigma^p I_4 + \tau^p I_5 + p\, \sigma \tau^{p-1} J_{p-1}
\]
and proceed as above.
\end{proof}

By \eqref{2.65} and Lemma \ref{Lemma 2.12},
\begin{equation} \label{3.22}
E_0(\sigma v + \tau w_{\eps,\delta/2 \theta}) \le E_0(\sigma v) + c_6\, \delta^\alpha \sigma^p + E_0(\tau w_{\eps,\delta/2 \theta}) + \frac{c_7}{p}\, (\eps/\delta)^{\alpha/(p-1)}\, \tau^p.
\end{equation}
Since $C \subset \Psi^{\lambda_k + c_1\, \delta^\alpha}$ by Lemma \ref{Lemma 2.10},
\begin{multline}
E_0(\sigma v) + c_6\, \delta^\alpha \sigma^p \le \sigma^p \left(\frac{1}{p} \int_{\R^{2N}} \frac{|v(x) - v(y)|^p}{|x - y|^{N+sp}}\, dx dy - \frac{\lambda}{p} \int_\Omega |v|^p\, dx + c_6\, \delta^\alpha\right)\\[10pt]
= \sigma^p \left(1 - \frac{\lambda}{\Psi(v)} + c_6\, \delta^\alpha\right) \le \sigma^p \left(1 - \frac{\lambda}{\lambda_k + c_1\, \delta^\alpha} + c_6\, \delta^\alpha\right) \le 0
\end{multline}
if $\delta > 0$ is sufficiently small. By \eqref{2.61},
\begin{multline} \label{3.24}
E_0(\tau w_{\eps,\delta/2 \theta}) + \frac{c_7}{p}\, (\eps/\delta)^{\alpha/(p-1)}\, \tau^p = \frac{\tau^p}{p}\, \bigg(\int_{\R^{2N}} \frac{|w_{\eps,\delta/2 \theta}(x) - w_{\eps,\delta/2 \theta}(y)|^p}{|x - y|^{N+sp}}\, dx dy\\[10pt]
- \lambda \int_\Omega w_{\eps,\delta/2 \theta}^p\, dx + c_7\, (\eps/\delta)^{\alpha/(p-1)}\bigg) - \frac{\tau^{p_s^\ast}}{p_s^\ast} \le \frac{s}{N}\, \bigg(\int_{\R^{2N}} \frac{|w_{\eps,\delta/2 \theta}(x) - w_{\eps,\delta/2 \theta}(y)|^p}{|x - y|^{N+sp}}\, dx dy\\[10pt]
- \lambda \int_\Omega w_{\eps,\delta/2 \theta}^p\, dx + c_7\, (\eps/\delta)^{\alpha/(p-1)}\bigg)^{N/sp}.
\end{multline}
Combining \eqref{3.22}--\eqref{3.24} with \eqref{2.62} and \eqref{2.63} gives
\[
\sup_{v \in C,\, \sigma, \tau \ge 0}\, E_0(\sigma v + \tau w_{\eps,\delta/2 \theta}) \le \begin{cases}
\dfrac{s}{N} \left(S + c_{15}\, (\eps/\delta)^{\alpha/(p-1)} - \lambda c_{16}\, \eps^{sp}\right)^{N/sp} & \text{if } N > sp^2\\[10pt]
\dfrac{s}{N}\, \big(S + c_{15}\, (\eps/\delta)^{sp} - \lambda c_{16}\, \eps^{sp} \abs{\log\, (\eps/\delta)}\big)^{N/sp} & \text{if } N = sp^2
\end{cases}
\]
for some constants $c_{15}, c_{16} > 0$, and \eqref{2.64} follows from this for sufficiently small $\eps > 0$. \qedsymbol

\bigbreak

\noindent{\bf \large Acknowledgement}

\medskip

\noindent This work was completed while the second author was visiting the Department of Mathematical Sciences at Florida Institute of Technology. She is grateful for the kind hospitality of the host institution.\\
The second author was partially supported by MIUR--PRIN project ``Qualitative and quantitative aspects of nonlinear PDEs'' (2017JPCAPN\underline{\ }005) and is member of the Research Group INdAM–GNAMPA.

\def\cdprime{$''$}


\begin{thebibliography}{10}

\bibitem{MR3927428}
Claudianor~O. Alves, Vincenzo Ambrosio, and Teresa Isernia.
\newblock Existence, multiplicity and concentration for a class of fractional
  {$p\&q$} {L}aplacian problems in {$\R^N$}.
\newblock {\em Commun. Pure Appl. Anal.}, 18(4):2009--2045, 2019.

\bibitem{MR3876875}
Vincenzo Ambrosio and Teresa Isernia.
\newblock On a fractional {$p\& q$} {L}aplacian problem with critical
  {S}obolev-{H}ardy exponents.
\newblock {\em Mediterr. J. Math.}, 15(6):Paper No. 219, 17, 2018.

\bibitem{MR3915604}
Vincenzo Ambrosio, Teresa Isernia, and Gaetano Siciliano.
\newblock On a fractional {$p\&q$} {L}aplacian problem with critical growth.
\newblock {\em Minimax Theory Appl.}, 4(1):1--19, 2019.

\bibitem{MR1741848}
Gianni Arioli and Filippo Gazzola.
\newblock Some results on {$p$}-{L}aplace equations with a critical growth
  term.
\newblock {\em Differential Integral Equations}, 11(2):311--326, 1998.

\bibitem{MR3910033}
Mousomi Bhakta and Debangana Mukherjee.
\newblock Multiplicity results for {$(p,q)$} fractional elliptic equations
  involving critical nonlinearities.
\newblock {\em Adv. Differential Equations}, 24(3-4):185--228, 2019.

\bibitem{MR4313576}
Stefano Biagi, Eugenio Vecchi, Serena Dipierro, and Enrico Valdinoci.
\newblock Semilinear elliptic equations involving mixed local and nonlocal
  operators.
\newblock {\em Proc. Roy. Soc. Edinburgh Sect. A}, 151(5):1611--1641, 2021.

\bibitem{MR3461371}
Lorenzo Brasco, Sunra Mosconi, and Marco Squassina.
\newblock Optimal decay of extremals for the fractional {S}obolev inequality.
\newblock {\em Calc. Var. Partial Differential Equations}, 55(2):Art. 23, 32,
  2016.

\bibitem{MR699419}
Ha{\"{\i}}m Br{\'e}zis and Elliott Lieb.
\newblock A relation between pointwise convergence of functions and convergence
  of functionals.
\newblock {\em Proc. Amer. Math. Soc.}, 88(3):486--490, 1983.

\bibitem{MR4160747}
Fanfan Chen and Yang Yang.
\newblock Existence of {S}olutions for the {F}ractional ({$p, q$})-{L}aplacian
  {P}roblems {I}nvolving a {C}ritical {S}obolev {E}xponent.
\newblock {\em Acta Math. Sci. Ser. B (Engl. Ed.)}, 40(6):1666--1678, 2020.

\bibitem{MR2514055}
Marco Degiovanni and Sergio Lancelotti.
\newblock Linking solutions for {$p$}-{L}aplace equations with nonlinearity at
  critical growth.
\newblock {\em J. Funct. Anal.}, 256(11):3643--3659, 2009.

\bibitem{MR4249816}
Serena Dipierro and Enrico Valdinoci.
\newblock Description of an ecological niche for a mixed local/nonlocal
  dispersal: an evolution equation and a new {N}eumann condition arising from
  the superposition of {B}rownian and {L}\'{e}vy processes.
\newblock {\em Phys. A}, 575:Paper No. 126052, 20, 2021.

\bibitem{MR956567}
Henrik Egnell.
\newblock Existence and nonexistence results for {$m$}-{L}aplace equations
  involving critical {S}obolev exponents.
\newblock {\em Arch. Rational Mech. Anal.}, 104(1):57--77, 1988.

\bibitem{MR0478189}
Edward~R. Fadell and Paul~H. Rabinowitz.
\newblock Generalized cohomological index theories for {L}ie group actions with
  an application to bifurcation questions for {H}amiltonian systems.
\newblock {\em Invent. Math.}, 45(2):139--174, 1978.

\bibitem{MR912211}
J.~P. Garc{\'{\i}}a~Azorero and I.~Peral~Alonso.
\newblock Existence and nonuniqueness for the {$p$}-{L}aplacian: nonlinear
  eigenvalues.
\newblock {\em Comm. Partial Differential Equations}, 12(12):1389--1430, 1987.

\bibitem{MR4142328}
Divya Goel, Deepak Kumar, and K.~Sreenadh.
\newblock Regularity and multiplicity results for fractional
  {$(p,q)$}-{L}aplacian equations.
\newblock {\em Commun. Contemp. Math.}, 22(8):1950065, 37, 2020.

\bibitem{MR1009077}
Mohammed Guedda and Laurent V{\'e}ron.
\newblock Quasilinear elliptic equations involving critical {S}obolev
  exponents.
\newblock {\em Nonlinear Anal.}, 13(8):879--902, 1989.

\bibitem{MR3357858}
Petru Mironescu and Winfried Sickel.
\newblock A {S}obolev non embedding.
\newblock {\em Atti Accad. Naz. Lincei Rend. Lincei Mat. Appl.},
  26(3):291--298, 2015.

\bibitem{MR3530213}
Sunra Mosconi, Kanishka Perera, Marco Squassina, and Yang Yang.
\newblock The {B}rezis-{N}irenberg problem for the fractional
  {$p$}-{L}aplacian.
\newblock {\em Calc. Var. Partial Differential Equations}, 55(4):Art. 105, 25,
  2016.

\bibitem{MR1998432}
Kanishka Perera.
\newblock Nontrivial critical groups in {$p$}-{L}aplacian problems via the
  {Y}ang index.
\newblock {\em Topol. Methods Nonlinear Anal.}, 21(2):301--309, 2003.

\bibitem{MR4293883}
Kanishka Perera.
\newblock An abstract critical point theorem with applications to elliptic
  problems with combined nonlinearities.
\newblock {\em Calc. Var. Partial Differential Equations}, 60(5):Paper No. 181,
  23, 2021.

\bibitem{MR2640827}
Kanishka Perera, Ravi~P. Agarwal, and Donal O'Regan.
\newblock {\em Morse theoretic aspects of {$p$}-{L}aplacian type operators},
  volume 161 of {\em Mathematical Surveys and Monographs}.
\newblock American Mathematical Society, Providence, RI, 2010.

\end{thebibliography}
\end{document}